\documentclass[11pt,reqno]{amsart}
\usepackage[utf8]{inputenc}
\usepackage{amsmath,amsthm,amssymb,amsopn}
\usepackage[pagebackref,colorlinks,linkcolor=red,citecolor=blue,urlcolor=blue,hypertexnames=false]{hyperref}
\usepackage{amsrefs}
\usepackage{amscd}

\input{xy}
\xyoption{all}
\newdir{ >}{{}*!/-5pt/@{>}}
\usepackage{tikz}
\usetikzlibrary{matrix,arrows}
\usepackage{txfonts}

\newcommand\restr[2]{{
  \left.\kern-\nulldelimiterspace 
  #1 
  \vphantom{\big|} 
  \right|_{#2} 
  }} 

\newcommand{\comment}[1]{}

\def\N{\mathbb{N}} 
\def\Z{\mathbb{Z}}

\def\C{\mathbb{C}}

\def\topdf{\texorpdfstring}
\theoremstyle{plain}
\newtheorem{teo}[equation]{Theorem} 
\newtheorem{thm}[equation]{Theorem}
\newtheorem{lema}[equation]{Lemma}
\newtheorem{lem}[equation]{Lemma}
\newtheorem{coro}[equation]{Corollary} 
\newtheorem{prop}[equation]{Proposition}

\theoremstyle{definition}

\newtheorem{ex}[equation]{Example}

\theoremstyle{remark} 

 \newtheorem{rem}[equation]{Remark}
 
  \numberwithin{equation}{section}

\newcommand{\cE}{\mathcal E}

\newcommand{\cR}{\mathcal R}

\newcommand{\cV}{\mathcal V}

\def\fB{\mathfrak{B}}

\def\fX{\mathfrak{X}}
\def\grp{\mathfrak{Grp}}

\def\Alpha{\Lambda}
\newcommand{\aha}{{{\rm Alg}_\ell}}
\newcommand{\lra}{\longrightarrow}
\newcommand{\iso}{\overset{\sim}{\lra}}
\newcommand{\simh}{\approx}
\newcommand{\onto}{\twoheadrightarrow}
\def\reg{\operatorname{reg}}
\def\sing{\operatorname{sing}}
\def\sink{\operatorname{sink}}
\def\inf{\operatorname{inf}}
\def\sour{\operatorname{sour}}
\def\triqui{\vartriangleleft}
\def\mspan{\operatorname{span}}
\def\supp{\operatorname{supp}}
\def\inc{\operatorname{inc}}

\def\diag{\operatorname{diag}}
\def\Gl{\operatorname{GL}}
\def\GL{\Gl}
\def\ad{\operatorname{ad}}
\def\Idem{\operatorname{Idem}}
\def\ev{\operatorname{ev}}
\def\id{\operatorname{id}}

\newcommand{\coker}{{\rm Coker}}
\renewcommand{\ker}{{\rm Ker}}

\DeclareMathOperator*{\colim}{colim}
\def\Ext{\operatorname{Ext}}
\def\cExt{\mathcal{E}xt}

\def\Hom{\operatorname{Hom}}
\def\End{\operatorname{End}}

\title{Homotopy classification of Leavitt path algebras}
\date{}

\author{Guillermo Corti\~nas}
\author{Diego Montero}
\email{gcorti@dm.uba.ar, dmontero@dm.uba.ar}
\urladdr{http://mate.dm.uba.ar/\~{}gcorti}
\address{Dep. Matem\'atica-IMAS, FCEyN-UBA\\ Ciudad Universitaria Pab 1\\
C1428EGA Buenos Aires\\ Argentina}

\thanks{Both authors were supported by CONICET and by grants UBACyT 20021030100481BA and  PICT 2013-0454. Corti\~nas research was supported by grant MTM2015-65764-C3-1-P (Feder funds).}

\begin{document}

\begin{abstract}
In this paper we address the classification problem for purely infinite simple Leavitt path algebras of finite graphs over a field $\ell$.
Each graph $E$ has associated a Leavitt path $\ell$-algebra $L(E)$. There is an open question which asks whether the pair $(K_0(L(E)), 
[1_{L(E)}])$, consisting of the Grothendieck group together with the class $[1_{L(E)}]$ of the identity, is a complete invariant for the classification, up to algebra isomorphism, of those Leavitt path algebras of finite graphs which are purely infinite simple. We show that $(K_0(L(E)), [1_{L(E)}])$ is a complete invariant for the classification of such algebras up to polynomial homotopy equivalence. To prove this we further develop the study of bivariant algebraic $K$-theory of Leavitt path algebras started in a previous paper and obtain several other results of independent interest.  
\end{abstract}

\maketitle

\section{Introduction} 

A directed graph $E$ consists of a set $E^0$ of vertices and a set $E^1$ of edges together with source and range functions $r,s:E^1\to E^0$. This article is concerned with the Leavitt path algebra $L(E)$ of a directed graph $E$ over a field $\ell$ (\cite{libro}). When $\ell=\C$, $L(E)$ is a normed algebra; its completion is the graph $C^*$-algebra $C^*(E)$.  A graph $E$ is called finite or countable if both $E^0$ and $E^1$ are finite or countable. A result of Cuntz and R{\o}rdam (\cite{ror}*{Theorem 6.5}) says that the purely infinite simple graph algebras associated to finite graphs, i.e. the purely infinite simple Cuntz-Krieger algebras, are classified up to (stable) isomorphism by the Grothendieck group $K_0$. It is an open question whether a similar result holds for Leavitt path algebras \cite{alps}. Here we prove that $K_0$ classifies simple Leavitt path algebras up to ($M_2$-) homotopy equivalence. In the following theorem and elsewhere, we use the following notations. We write
$\iota_2:R\to M_2R$ for the inclusion of an algebra into the upper left hand corner of the matrix algebra, $\phi\simh \psi$ to indicate that two algebra homomorphisms $\phi$ and $\psi$ are (polynomially) homotopic and $\phi\simh_{M_2}\psi$ to mean that $\iota_2\phi\simh\iota_2\psi$. We also put $[1_R]$ for the $K_0$-class of the identity of a unital algebra $R$. In Theorem \ref{thm:main2} we prove the following.

\begin{thm}\label{intro:main}
Let $E$ and $F$ be finite graphs. Assume that $L(E)$ and $L(F)$ are purely infinite simple. Let 
$\xi:K_0(L(E))\to K_0(L(F))$ be an isomorphism of groups. Then there exist nonzero algebra homomorphisms $\phi:L(E)\leftrightarrow L(F):\psi$ such that 
$K_0(\phi)=\xi$, $K_0(\psi)=\xi^{-1}$, $\psi\phi\simh_{M_2} \id_{L(E)}$ and $\phi\psi\simh_{M_2}\id_{L(F)}$. If moreover $\xi([1_{L(E)}])=[1_{L(F)}]$ then $\phi$ and $\psi$ can be chosen to be unital homomorphisms such that $\psi\phi\simh \id_{L(E)}$ and $\phi\psi\simh \id_{L(F)}$. 
\end{thm}

We also prove other results which we think are of independent interest. For example we have the following embedding theorem, proved in  Corollary \ref{coro:tododentro}.

\begin{thm}\label{thm:embed} Let $E$ be a graph such that $L(E)$ is simple and let $R$ be a unital purely infinite algebra. 
\item[i)] If $E$ is countable then $L(E)$ embeds as a subalgebra of $M_\infty R$.
\item[ii)] If $E$ is finite and $[1_R]=0$ in $K_0(R)$, then $L(E)$ embeds as a unital subalgebra of $R$. 
\item[iii)] If $E$ is finite, then $L(E)$ embeds as a subalgebra of $R$. 
\end{thm}
For particular $R$, we have the following result on uniqueness up to homotopy for embeddings into $R$. In the next theorem and elsewhere, we write $[A,R]$ and 
$[A,R]_{M_2}$ for the set of homotopy classes and $M_2$-homotopy classes of homomorphisms $A\to R$. If moreover, $A$ and $R$ are unital, we write $[A,R]_1$ for the set of homotopy classes of unital homomorphisms $A\to R$. In the next theorem and elsewhere we use the notion of regular supercoherent ring from \cite{gersten}. For example, $L(E)$ is regular supercoherent for every finite graph $E$ (\cite{libro}*{Lemma 6.4.16}). We write $L_n$ for the Leavitt path algebra of the one-vertex graph with $n$ loops.

\begin{thm}\label{thm:embel2}
Let $E$ be finite graph such that $L(E)$ is simple and $R$ a purely infinite simple, regular supercoherent unital algebra. Then  \  $[L(E),L_2]_1= [L(E),L_2]_{M_2}\setminus\{0\}$, $[L(E),R\otimes L_2]_1=[L(E),R\otimes L_2]_{M_2}\setminus\{0\}$, and both sets have exactly one element each.  
\end{thm}

In particular, Theorem \ref{thm:embel2} implies that if $d:L_2\to L_2\otimes L_2$, $d(x)=1\otimes x$ and  $\phi: L_2\to L_2\otimes L_2$ is a nonzero homomorphism, then $\phi\simh_{M_2}d$ and that if $\phi$ is unital then $\phi\simh d$.

In \cite{dwkk} we introduced, for an algebra $A$ and a unital algebra $R$, an abelian monoid of homotopy classes of extensions of $A$ by $M_\infty R$, and considered its group completion $\cExt(A,R)$. We showed in \cite{dwkk}*{Remark 5.8} that if $E$ is a graph such that $E^0$ is finite and $E^1$ is countable, then for the bivariant algebraic $K$-theory group $kk_n(A,R)$ of \cite{kkwt}, there is a natural map
\begin{equation}\label{intro:mapextkk}
\cExt(L(E),R)\to kk_{-1}(L(E),  R).
\end{equation}
Recall that a ring $R$ is \emph{$K_n$-regular} if the canonical map $K_n(R)\to K_n(R[t_1,\dots,t_m])$ is an isomorphism for every $m$. For example, every Leavitt path algebra is $K_n$-regular for all $n\in\Z$, by \cite{dwkk}*{Example 5.5}. 
 
\begin{thm}\label{intro:ext}
Let $E$ be a finite graph such that $L(E)$ is simple. Let $R$ be either a division algebra or a $K_0$-regular purely infinite simple unital algebra. Then the natural map \eqref{intro:mapextkk} is an isomorphism 
\[
\cExt(L(E),R)\iso kk_{-1}(L(E),R).
\]
\end{thm}
Combining Theorem \ref{intro:ext} with results from our previous paper \cite{dwkk} we are able to compute $\cExt(L(E), R)$ in some cases. For example, we get that if $E$ is as in Theorem \ref{intro:ext}, $\reg(E)=E^0\setminus\sink(E)$ and $I$ and $A_E\in\Z^{\reg(E)\times E^0}$ are the identity and the incidence matrices with the rows corresponding to sinks removed, then
\begin{equation}\label{intro:ckext}
\cExt(L(E),\ell)=\coker(I-A_E).
\end{equation} 
If moreover $K_0(L(E))$ is torsion, then for every $R$ as in Theorem \ref{intro:ext} (in particular, for $R=\ell$ and for every purely infinite simple unital Leavitt path algebra $R$), we have 
\begin{equation}\label{intro:extext}
\cExt(L(E),R)=\Ext^1_\Z(K_0(L(E)),K_0(R)).
\end{equation}

The following theorem is the main technical result of the paper; it is key for the proof of Theorems \ref{intro:main}, \ref{thm:embel2} and \ref{intro:ext}.

\begin{thm}\label{intro:kklift}
Let $E$ be a finite graph such that $L(E)$ is simple and $R$ a purely infinite simple unital algebra. Assume that $R$ is $K_1$-regular. Then
the canonical map
\[
[L(E),R]_{M_2}\setminus\{0\}\to kk(L(E),R)
\]
is an isomorphism of monoids. 
\end{thm}
Thanks to Remark \ref{rem:gagp}, we may view Theorem \ref{intro:kklift} as a generalization of the theorem of Ara, Goodearl and Pardo \cite{agp} which says that if $R$ is as in the theorem and $\cV(R)$ is the monoid of Murray-von Neumann equivalence classes of idempotent matrices in $M_\infty R$, then $K_0(R)=\cV(R)\setminus\{0\}$. In fact, the latter result is used in the proof of  Theorem \ref{intro:kklift}.
 The proof of Theorem \ref{intro:kklift} also uses results from our previous paper on $kk$ of Leavitt path algebras and an adaptation to the purely algebraic setting, developed in Sections  \ref{sec:k0k1lift} and \ref{sec:kklift}, of several results proved for the $C^*$-algebra setting in R\o rdam's article \cite{ror}. 

\goodbreak

\bigskip

The rest of this paper is organized as follows. In Section \ref{sec:v1} we prove (Corollary \ref{coro:kv1pis}) that if $R$ is a $K_1$-regular, purely infinite simple and unital algebra, then $K_1(R)$ is isomorphic to the group $\pi_0(U(R))$ of polynomially connected components of the group of invertible elements of $R$. The case of Theorem \ref{intro:kklift} when $L(E)$ is not purely infinite is contained in Proposition  \ref{prop:homotopis} (see Remark \ref{rem:simpnopis}). Section \ref{sec:k0lift} considers the problem of whether a given group homomorphism $K_0(L(E))\to K_0(R)$ can be lifted to an algebra homomorphism. We show in Theorem \ref{thm:k0liftr} that if $R$ is purely infinite simple and unital and $E$ is countable, then any group homomorphism $\xi:K_0(L(E))\to K_0(R)$ is induced by an algebra homomorphism $\psi:L(E)\to M_\infty R$, that if moreover $E^0$ is finite and $\xi$ is unital (i.e. $\xi([1_{L(E)}])=[1_R])$ then $\xi$ is also induced by a unital homomorphism 
$\phi:L(E)\to R$, and that if $E$ is finite then any group homomorphism
$\xi:K_0(L(E))\to K_0(R)$ is induced by a nonzero algebra homomorphism $\phi:L(E)\to L(F)$. Theorem \ref{thm:embed} is Corollary \ref{coro:tododentro}. If $E$ is a finite graph with reduced incidence matrix $A_E$ as above, we shall abuse notation and write 
$I-A_E^t$ for the transpose of the matrix of \eqref{intro:ckext}. Section \ref{sec:k0k1lift} is concerned with the question of whether, given a finite graph $E$ with reduced incidence matrix $A_E$, an algebra $R$ and a pair $(\xi_0,\xi_1)$ of group homomorphisms $\xi_0:K_0(L(E))\to K_0(R)$ and $\xi_1:\ker(I-A_E^t)\to K_1(R)$, there is an algebra homomorphism simultaneously inducing $\xi_0$ and $\xi_1$. We prove in Theorem \ref{thm:ror} that if $L(E)$ is simple and $R$ is purely infinite, unital and $K_1$-regular, then there is an algebra homomorphism $\phi:L(E)\to R$
which induces both $\xi_0$ and $\xi_1$, and that if $\xi_0$ is unital, then $\phi$ can be chosen to be a unital homomorphism $L(E)\to R$. 
 Section \ref{sec:kklift} is devoted to the proof of Theorem \ref{thm:kkliftr}, which contains the case of Theorem \ref{intro:kklift} when $L(E)$ is purely infinite simple. Theorem \ref{intro:main} is proved in Section \ref{sec:main} (Theorem \ref{thm:main2}). Section \ref{sec:ext} is devoted to algebra extensions. Theorem \ref{intro:ext} is contained in Theorem \ref{thm:ext}; formulas \eqref{intro:ckext} and \eqref{intro:extext} are proved in Corollary \ref{coro:ckext} and Example \ref{ex:cute}. Section \ref{sec:tenso} is concerned with maps to $L_2$ and $R \otimes L_2$; Theorem \ref{thm:embel2} is proved in Theorem \ref{thm:mapl2}. 
  
\goodbreak
\medskip
\noindent{\it Acknowledgements. } A previous version of this article contained a proof of Proposition \ref{prop:vle} for the particular case when $R$ is a Leavitt path algebra. We are indebted to Pere Ara for pointing out that in fact the result holds for every unital purely infinite ring $R$, and that the proof is immediate from results in \cite{agp}.

\section{Idempotents, units and the groups \topdf{$K_0$}{K0} and \topdf{$K_1$}{K1} in the purely infinite simple unital case}\label{sec:v1}

Let $R$ be a ring; write $\Idem(R)$ for the set of idempotent elements. Let $p,q\in\Idem(R)$. We write $p\sim q$ if $p$ and $q$ are \emph{Murray-von Neumann equivalent} \cite{agp}; that is, if there exist elements $x\in pRq$
and $y\in qRp$ such that $xy=p$ and $yx=q$. We call such pair $(x,y)$ an \emph{MvN equivalence} from $p$ to $q$ and write $(x,y):p\sim q$. 

Put $\Idem_n(R)=\Idem(M_n(R))$, $1\le n\le \infty$. If $R$ is unital, we write 
\[
\cV_n(R)=\Idem_n(R)/\sim \qquad (1\le n<\infty),\quad \cV(R)=\Idem_\infty(R)/\sim.
\]

\begin{rem}\label{rem:vr} One may also define $\cV(R)$ as the set of isomorphism classes of finitely generated projective right modules. The equivalence between the two definitions follows from \cite{rosen}*{Theorem 1.2.3} and \cite{black}*{Propositions 4.2.5 and 4.3.1}. One checks that if $f:R\to S$ is a homomorphism and $f(1)=p$, then 
under the identification, the map $\cV(R)\to \cV(S)$ induced by $M_\infty R\to M_\infty S$ corresponds to the scalar extension functor $\otimes_RpS$.
\end{rem}

If $p,q\in\Idem(R)$ and $pq=qp=0$ we say that $p$ and $q$ are \emph{orthogonal} and write $p\perp q$ to indicate this. An idempotent $p$ in a ring $R$ is \emph{infinite}
if there exist orthogonal idempotents $q,r\in R$ such that $p= q+r$, 
$p\sim q$ and $r\ne 0$. A ring $R$ is said to be \emph{purely infinite simple} if for every nonzero element $x\in R$ there exist $s,t\in R$ such that $sxt$ is an infinite idempotent. If $R$ is unital this is equivalent to asking that $R$ not be a division ring and that for every $x\in R$ there are $a,b\in R$ such that $axb=1$.

The following theorem describing $K_0$ and $K_1$ of purely infinite simple unital rings is due to Ara, Goodearl and Pardo.
 If $R$ is a unital ring, write $U(R)$ for the group of invertible elements of $R$.

\begin{thm}\label{thm:kpis}\cite{agp}*{Corollary 2.3 and Theorem 2.4}
If $R$ is a purely infinite simple unital ring, then $$K_0(R) = \mathcal{V}(R) \backslash \{ [0] \}$$
$$K_1(R) = U(R)^{ab}.$$
\end{thm}

\begin{prop}\label{prop:vle} Let $R$ be a purely infinite simple unital ring. Then the map
$\iota:\cV_1(R)\to \cV(R)$ is an isomorphism. Moreover, for every $n\ge 1$ and every element $(q_1,\dots,q_n)\in \Idem_\infty(R)^n$ there exists $(p_1,\dots,p_n)\in\Idem_1(R)^n$, such that $p_i\sim q_i$ in $\Idem_\infty(R)$ and such that $p_i\perp p_j$ for $i\ne j$.
\end{prop}
\begin{proof}
This is straightforward from \cite{agp}*{Proposition 1.5 and Lemma 1.1}
\end{proof}

Combining Proposition \ref{prop:vle} and Theorem \ref{thm:kpis} we obtain the following.

\begin{coro}\label{coro:k0pis}
Let $R$ be a purely infinite simple unital ring. Then 
$$ K_0(R) \cong \cV_1(R)\backslash \{[0]\}.  $$ 
\end{coro}

\begin{coro}\label{coro:equiv} Let $R$ be a purely infinite simple unital ring and let  $e, f\in R$ be nonzero idempotents. 
Then the following are equivalent
\item[(1)] $e \sim f$.
\item[(2)] $[e] = [f] $ in $K_0(R)$. 

\noindent If furthermore $e,f\in\Idem_1(R)\backslash\{0,1\}$ then the above conditions are also equivalent to the following. 

\item[(3)] There exists $ u \in U(R)$ such that $ f= u e u^{-1}$.
\item[(4)] There exists a commutator $ u \in [U(R):U(R)]$ such that $ f= u e u^{-1}$.
\end{coro}

\begin{proof}
The equivalence of (1) and (2) follows from Corollary \ref{coro:k0pis}. By \cite{black}*{Proposition 4.2.5}, (3) is equivalent to having simultaneously $e \sim f$ and $1-e \sim 1-f$. Hence to prove that (1) implies (3) it only remains to show that $1-e \sim 1-f$. But 
$$ [e] + [1-e] = [1]= [f]+[1-f] $$
in $K_0(R)$ and  $[e]= [f]$, implies  $[1-e] = [1-f]$ in $K_0(R)$ and therefore in $\cV_1(R)$. Hence  $1-e \sim 1-f$.
Next we show that (3) implies (4).  Because $R$ is simple and $f\ne 1$, $1-f$ is a full idempotent. Hence $ (1-f) L(E) (1-f)$ is purely infinite simple (by \cite{agp}*{Corollary 1.7}) and the inclusion induces an isomorphism $ K_1((1-f) R (1-f)) \iso K_1(R)$. By Theorem \ref{thm:kpis}, this implies that the induced map $U((1-f) R (1-f))^{ab}\to U(R)^{ab}$ is an isomorphism. Since the latter map sends $[\xi]\mapsto [\xi+f]$, there is an element $\omega\in U((1-f) R (1-f))$ such that $ [\omega+f]= [u^{-1}]$. Then $(\omega+f) u  \in [U(R):U(R)]$ and $(\omega+f) u e u^{-1} (\omega^{-1} +f ) = f$.
To prove that (4) implies (1) take $x = e u^{-1} f$ and $y = f u e$; we have $xy=e$ and $yx=f$. 
\end{proof}

Let $G:\aha\to\grp$ be a functor from algebras to groups and let $A\in\aha$. The \emph{connected component} of $G(A)$ is the subgroup
\[
G(A)\supset G(A)^0=\{g\mid (\exists u(t)\in G(A[t]))\quad u(0)=1, u(1)=g\}.
\]
Observe that $G(A)^0$ is a normal subgroup. We write 
\[
\pi_0G(A)=G(A)/G(A)^0.
\]
The \emph{Karoubi-Villamayor} $K_1$-group (\cite{kv}) is 
\[
KV_1(A)=\pi_0(\GL(A)).
\]
Observe that every elementary matrix is in $\GL(A)^0$. It follows that we have a surjective homomorphism
\begin{equation}\label{map:k1kv1}
K_1(A)\onto KV_1(A).
\end{equation}
 By \cite{weih}*{Proposition 1.5}, the map \eqref{map:k1kv1} is an isomorphism whenever $A$ is $K_1$-regular. 

\begin{lem}\label{lem:kv1pis}
Let $R$ be a unital ring.
\item[i)] If $p\in \Idem(R)$ and $u\in U(pRp)^0$, then $u+1-p\in U(R)^0$.
\item[ii)] Let $x_1,\dots,x_n,y_1,\dots,y_n\in R$ such that $y_ix_jy_i=\delta_{i,j}y_i$, $x_iy_jx_i=\delta_{i,j}x_i$. Set $p_i=x_iy_i$, $q_i=y_ix_i$, $P=\bigoplus_{i=1}^np_iR$, $Q=\bigoplus_{i=1}^nq_iR$. Then the map 
\begin{gather*}
c_{y,x}:=\End_R(P)=\bigoplus_{i,j}p_jRp_i\to \bigoplus_{i,j}q_jRq_i=\End_R(Q),\\
a\mapsto \diag(y_1,\dots,y_n)a\diag(x_1,\dots,x_n)
\end{gather*}
is an isomorphism which sends $U(\End_R(P))^0$ isomorphically onto $U(\End_R(Q))^0$.
\end{lem}
\begin{proof} Straightforward. \end{proof}
\begin{prop}\label{prop:kv1pis}
Let $R$ be a unital purely infinite simple ring. Then the canonical map $\pi_0(U(R))\to \pi_0(\GL(R))=KV_1(R)$ is an isomorphism. 
\end{prop}
\begin{proof}
We know from Theorem \ref{thm:kpis} and \eqref{map:k1kv1} that $U(R)\to KV_1(R)$ is surjective. The kernel of this map is $U(R)\cap \GL(R)^0$; it is clear that it contains $U(R)^0$. We have to show that 
\begin{equation}\label{u0gl0}
U(R)\cap\GL(R)^0\subset U(R)^0.
\end{equation} 
We claim that the argument of the proof  that $[\GL(R):\GL(R)]\cap U(R)\subset [U(R):U(R)]$ in \cite{agp}*{Theorem 2.3} can be adapted to prove \eqref{u0gl0}.
The proof in \emph{loc.cit.} has two parts. The first part shows that if $0\ne p\in \Idem(R)$ and $u\in [\GL(R):\GL(R)]\cap U(R)$ satisfies 
\begin{equation}\label{upup}
u=p+(1-p)u(1-p)
\end{equation} 
then  $u\in [U(R):U(R)]$. Using the same argument and taking Lemma \ref{lem:kv1pis} into account, one shows that if \eqref{upup} is in $\GL(R)^0$, then it must be in $U(R)^0$. In the second part of the proof of \cite{agp}*{Theorem 2.3} it is observed that for adequately chosen  idempotents $e$ and $f\in T=eRe$ and elements $x_1,y_1,\dots,x_n,y_n\in R$, the assignment $a\mapsto \diag(y_1,\dots,y_n)a\diag(x_1,\dots,x_n)$ induces an isomorphism between $R$ and the subring
\[
 M_n(T)\supset S=\{(a_{i,j}):a_{i,n}\in Tf, a_{n,i}\in fT \text{ for all } 1\le i\le n\}.
\] 
Let $\cE\subset U(R)$ be the image under the isomorphism $U(S)\iso U(R)$ of the subgroup generated by the set of those elementary matrices $1+a\epsilon_{i,j}$ $i\ne j$ which are elements of $S$. The authors then proceed, using the argument of the proof of \cite{m&m}*{Theorem 2.2}, to show that any $u\in U(R)$ is congruent modulo $\cE$ to one of the form of \eqref{upup}. In view of Lemma \ref{lem:kv1pis} and of the fact that elementary matrices above are in $U(S)^0$, this shows that any $u\in U(R)$ is congruent modulo $U(R)^0$ to one of the form \eqref{upup}. This finishes the proof.
\end{proof}

\begin{coro}\label{coro:kv1pis}
If $R$ is unital, purely infinite simple and $K_1$-regular then $K_1(R)=\pi_0(U(R))$. 
\end{coro}

Let $A$ be an algebra. Identify $\Hom_\aha(\ell,A)=\Idem_1(A)$ via the bijection $\phi\mapsto \phi(1)$. We say that two idempotents 
$p,q\in\Idem_1(A)$ are \emph{homotopic}, and write $p\simh  q$, if the corresponding homomorphisms $\ell\to A$ are homotopic.

\begin{lem}\label{lem:idemzero}
Let $A$ be an algebra and $p\in \Idem_1(A)$. Then $p\simh  0$ if and only if $p=0$. If $A$ is unital, then $p\simh  1$ if and only if 
$p=1$.\end{lem}
\begin{proof} The if part of both assertions is clear. One checks that if $x\in\{0,1\}$ and $p(t)\in\Idem_1(A[t])$ satisfies $p(0)=x$, 
then $p=x$. The only if part of both assertions follows from this.
\end{proof}

Recall (see \cite{dwkk}*{Section 2}) that a \emph{$C_2$-algebra} is a unital algebra $R$ together with a unital homomorphism from the Cohn algebra $C_2$ to $R$. Thus a $C_2$-algebra is a unital algebra together with elements $x_1,x_2,y_1,y_2\in R$ such that $y_ix_j=\delta_{i,j}$. For example, if $R$ is a purely infinite simple unital algebra then $R$ is a $C_2$-algebra (see \cite{agp}*{Proposition 1.5}). Put
\begin{equation}\label{map:boxplus}
\boxplus:R\oplus R\to R, \quad a\boxplus b=x_1ay_1+x_2by_2.
\end{equation}

\begin{lem}\label{lem:tensosum}
Let $R_1$ and $R_2$ be $C_2$-algebras and let $A_1\triqui R_1$ and $A_2\triqui R_2$ ideals. Let $\boxplus_i:A_i\oplus A_i\to A_i$ be the sum operation \eqref{map:boxplus}. Then the maps 
\[
\boxplus_1\otimes\id_{A_2},\id_{A_1}\otimes\boxplus_2:A_1\otimes A_2\oplus A_1\otimes A_2\to A_1\otimes A_2
\]
are $M_2$-homotopic.   
\end{lem}
\begin{proof} Straightforward from \cite{dwkk}*{Lemma 2.3}.
\end{proof}

Let $C$ be an algebra, $A,B\subset C$ subalgebras and $x,y\in C$ satisfying 
$xAy\subset B$ and $ayxa'=aa'$  $(a,a'\in A)$; then the following map is an algebra homomorphism 
\begin{equation}\label{map:adxy}
\ad(x,y):A\to B,\quad \ad(x,y)(a)=xay.
\end{equation}
If $C$ is unital and $y=x^{-1}$, then $\ad(x,y)=\ad(x)$ is the usual conjugation map.

\begin{lem}\label{lem:adhomo}
Let $A$ and $R$ be algebras, with $A$ finitely generated.  Then:

\item[i)]
The canonical map
\[
[A,M_\infty R]\to [A,M_\infty R]_{M_2}
\]
is bijective.

\item[ii)] If furthermore $R$ is a $C_2$-algebra then the canonical map
\[
[A,R]_{M_2}\to [A,M_\infty R]_{M_2}
\]
is an isomorphism of monoids. 
\end{lem}
\begin{proof} 
\item[i)] Because $A$ is finitely generated, 
\[
[A,M_\infty R]=\colim_n[A, M_{2^n}R]=\colim_n[A, M_{2^n}R]_{M_2}=[A,M_\infty R]_{M_2}.
\] 
\item[ii)]
Because $R$ is an $C_2$-algebra, the map $[A,R]_{M_2}\to [A,M_\infty R]_{M_2}$ is a monoid homomorphism by Lemma \ref{lem:tensosum}. We have to prove that it is bijective. Observe that $M_2R$ is again a $C_2$-algebra. Hence in view of the proof of part i), it suffices to show that
$[A,R]_{M_2}\to [A,M_2R]_{M_2}$ is bijective. Let $x=\epsilon_{1,1}\otimes x_1+\epsilon_{1,2}\otimes x_2$ and $y=\epsilon_{1,1}\otimes y_1
+\epsilon_{2,1}\otimes y_2$. By \cite{dwkk}*{Lemma 2.3}, the following diagram is $M_2$-homotopy commutative
\[
\xymatrix{M_2R\ar@{=}[dr]\ar[r]^{\ad(x,y)}&\iota_2(R)\ar[d]^{\inc}\\
                                          & M_2R.}
\]
It follows that the map of ii) is surjective. Injectivity follows similarly. 
\end{proof}

\begin{lem}\label{lem:adxy}
Let $\phi,\psi:A\to R$ be algebra homomorphisms with $R$ unital. Assume that there are $n\ge 1$ and $u\in\Gl_n(R)$ such that $\ad(u)\iota_n\phi=\iota_n\psi$. 
Then there are elements $x,y\in R$ such that $\ad(x,y)\phi=\psi$. If moreover $A$, $\phi$ and $\psi$ are unital, then we may choose $x$ invertible and $y=x^{-1}$. 
\end{lem}
\begin{proof} Put $v=u^{-1}$. It follows from the identity $\ad(u)\iota_n\phi=\iota_n\psi$ that for every $a\in A$, $u_{1,1}\phi(a)v_{1,1}=\psi(a)$
and $u_{i,1}\phi(a)=\phi(a)u_{1,i}=0$ if $i\ne 1$. Hence $x=u_{1,1}$ and $y=v_{1,1}$ satisfy $\ad(x,y)\phi=\psi$ and if $\phi$
and $\psi$ are unital, then $xy=yx=1$.
\end{proof}

\begin{prop}\label{prop:homotopis}
Let $R$ be a unital, purely infinite simple, $K_0$-regular algebra and $n\ge 1$. Then the natural monoid maps
\[
[M_n,R]_{M_2}\to [M_n,M_\infty R]\setminus\{0\}\to kk(M_n,R)\cong kk(\ell,R)\cong K_0(R)
\]
are bijective. Moreover, for nonzero algebra homomorphisms $M_n\to M_\infty R$ as well as for unital algebra homomorphisms $M_n\to R$, being  homotopic is the same as being conjugate. 
\end{prop}
\begin{proof} Because as explained above, any purely infinite simple unital algebra is a $C_2$-algebra, the map $[M_n,R]_{M_2}\to [M_n,M_\infty R]$ is an isomorphism of monoids by Lemma \ref{lem:adhomo}.  Since $(\iota_n)^*:kk(M_n,R)\to kk(\ell,R)=K_0R$ is an isomorphism, to prove that the map $[M_n,M_\infty R]\setminus\{0\}\iso kk(M_n,R)$ is surjective, it suffices, by Corollary \ref{coro:k0pis}, to show that the image of its composite with $\iota_n^*$ contains the class of every nonzero idempotent in $R$. Let $p\in \Idem_1 R\setminus\{0\}$; by Proposition \ref{prop:vle} we may choose $q\in \Idem_1R$
, $q\sim p$, and an embedding $\theta:M_n\to R$ sending $\epsilon_{1,1}\to q$. Hence the map of the proposition is surjective. If two homomorphisms $\phi,\psi\in\Hom_{\aha}(M_n,M_\infty R)$ induce the same $K_0$-element then they are conjugate by the argument of the proof of \cite{goodearl}*{Lemma 15.23(b)}, and therefore homotopic by Lemma \ref{lem:adhomo} and \cite{dwkk}*{Lemma 2.3}. From what we have just proved and Lemma \ref{lem:adxy}, it follows that if two unital homomorphisms $M_n\to R$ are homotopic then they are conjugate. This finishes the proof. 
\end{proof}

\begin{rem}\label{rem:simpnopis} Let $E$ be a finite graph such that $L(E)$ is simple. If $L(E)$ is not purely infinite, then it follows from  
\cite{libro}*{Lemma 2.9.5} and source elimination \cite{libro}*{Definition 6.3.26} that $L(E)\cong M_n$ for some $1\le n<\infty$. Hence, since $K_n$-regularity implies $K_{n-1}$-regularity \cite{vorst}, 
Proposition \ref{prop:homotopis} implies Theorem \ref{intro:kklift} in the case when $L(E)$ is simple and not pure infinite. 
\end{rem}

\section{Lifting \topdf{$K$}{K}-theory maps to algebra maps:\topdf{$K_0$}{K0}}\label{sec:k0lift}

Recall that a vertex $v\in E^0$ is \emph{singular} if it is either a sink or an infinite emitter, and that it is \emph{regular} otherwise. We write $\reg(E)$, $\sink(E)$, $\sour(E)$ and $\inf(E)$ for the sets of regular vertices, sinks, sources, and infinite emitters, and put $\sing(E)=\sink(E)\cup\inf(E)$.

Let $R$ and $S$ be unital algebras and $\xi:K_0(R)\to K_0(S)$. We call $\xi$ \emph{unital} if $\xi([1_R])=[1_S]$.

\begin{thm}\label{thm:k0liftr}
Let $E$ be a graph, $R$ a purely infinite simple unital algebra, and $\xi:K_0(L(E))\to K_0(R)$ a group homomorphism. Set $\iota:R\to M_\infty(R)$, $\iota(a)=\epsilon_{1,1}\otimes a$.
\item[i)] If $E$ is countable, then there exists a nonzero algebra homomorphism $\psi:L(E)\to M_\infty R$ such that $K_0(\psi)=K_0(\iota)\xi$. 
\item[ii)] If $E$ is finite, then there exists a nonzero algebra homomorphism $\psi:L(E)\to R$ such that $K_0(\psi)=\xi$.
\item[iii)] If $E^0$ is finite, $E^1$ countable and $\xi$ unital, then there is a unital homomorphism $\phi:L(E)\to R$ such that 
$K_0(\phi)=\xi$.
\end{thm}
\begin{proof}
Assume first that $E$ is countable and row-finite. By Theorem \ref{thm:kpis} there are orthogonal idempotents $\{p_e:e\in E^1\}\cup\{p_v:v\in\sing(E)\}\subset \Idem_\infty(R)\setminus\{0\}$ such that $[p_v]=\xi[v]$ and $[p_e]=\xi[ee^*]$ in $K_0(R)$ $(v\in \sink(E)$, $e\in E^1$). If $e\in E^1$ and $r(e)\in\reg(E)$ then  
\[
[p_e]=[\sum_{f\in E^1, s(f)=r(e)}p_f].
\]
Hence for $\sigma_f=\sum_{f\in E^1, s(f)=r(e)}p_f$ there are elements $x_e,y_e\in M_\infty(R)$ implementing an MvN equivalence $p_e\sim\sigma_e$. Similarly if $e\in E^1$ and $r(e)=v\in\sink(E)$, then there is an MvN equivalence $(x_e,y_e):p_e\sim p_v$ with $x_e, y_e\in M_\infty R$. One checks that the prescriptions 
\[
\psi(e)=x_e,\psi(e^*)=y_e\quad (e\in E^1),\quad \psi(v)=p_v\quad (v\in\sink(E))
\]
define a nonzero algebra homomorphism $\psi:L(E)\to M_\infty R$. Let $\tau:M_\infty M_\infty\to M_\infty M_\infty$, $\tau(x\otimes y)=y\otimes x$; one checks that $\tau\otimes Id_R$ induces the identity of $K_0(M_\infty R)$. By construction $K_0(\psi)$ agrees with $K_0(\tau\otimes 1)K_0(\iota)\xi=K_0(\iota)\xi$ on the classes of those vertices which are sinks and on those of  elements of the form $ee^*$ $(e\in E^1)$. Since the latter generate $K_0(L(E))$ (by \cite{libro}*{Theorem 3.2.5}), we have $K_0(\psi)=K_0(\iota)\xi$. 
 
For general countable $E$, let $E_\delta$ be a desingularization and $f:L(E)\to L(E_\delta)$ the canonical homomorphism \cite{aap}*{Section ~5}; then $K_0(f)$ is an isomorphism. Hence by what we have just proved, there exists an algebra homomorphism $\psi':L(E_\delta)\to M_\infty(R)$ such that $K_0(\psi')=K_0(\iota)\xi K_0(f)^{-1}$. Then
$\phi=\psi'f$ satisfies $K_0(\psi)=K_0(\iota)\xi$. This proves i). Next assume that $E^1$ is countable, that $E^0$ is finite and that $\xi([1_{L(E)}]=[1_R]$. Let $\psi:L(E)\to M_\infty(R)$ be a homomorphism such that $K_0(\iota)\xi=K_0(\psi)$. Set $p=\psi(1)$; then $\psi(L(E))\subset pM_\infty (R)p$ and there is an MvN equivalence $(x,y):p\sim \epsilon_{1,1}$. It follows that there is a unique unital homomorphism $\phi:L(E)\to R$ such that $\iota\phi=\ad(y,x)\psi$. By \cite{dwkk}*{Lemma 2.3}, $\phi$ satisfies the requirements of iii). Finally assume that $E$ is finite. By Corollary \ref{coro:k0pis} and Proposition \ref{prop:vle} there are orthogonal idempotents $\{p_e:e\in E^1\}\cup\{p_v:v\in\sink(E)\}\subset \Idem_1(R)\setminus\{0\}$ such that $[p_v]=\xi[v]$ and $[p_e]=\xi[ee^*]$ $(v\in \sink(E)$, $e\in E^1$). If $e\in E^1$ and $r(e)\notin\sink(E)$ then by Corollary \ref{coro:k0pis}, for $\sigma_e$ as in the proof of Theorem \ref{thm:k0liftr} there are elements $x_e\in p_e R \sigma_e$ and $y_e\in\sigma_e R p_e$ such that $p_e=x_ey_e$ and $\sigma_e=y_ex_e$. Similarly, if $e\in E^1$ and $r(e)=v\in\sink(E)$, then there are $x_e\in p_e R p_v$ and $y_e\in p_v R p_e$ such that $y_ex_e=p_v$ and $x_ey_e=p_e$. One checks that the prescriptions 
\[
\psi(e)=x_e,\psi(e^*)=y_e\quad (e\in E^1),\quad \psi(v)=p_v\quad (v\in\sink(E))
\]
define a nonzero algebra homomorphism $\psi:L(E)\to R$ such that $K_0(\psi)=\xi$.   
\end{proof}

\begin{coro}\label{coro:tododentro}
Let $R$ be a unital purely infinite algebra and $E$ a graph such that $L(E)$ is simple. 

\item[i)] If $E$ is countable, then $L(E)$ embeds as a subalgebra of $M_\infty R$.
 
\item[ii)] If $E^1$ is countable, $E^0$ is finite and $[1_R]=0$ in $K_0(R)$, then $L(E)$ embeds as a unital subalgebra
of $R$. 

\item[iii)] If $E$ is finite then $L(E)$ embeds as a subalgebra of $R$. 
\end{coro}

\begin{proof} Apply Theorem \ref{thm:k0liftr} to the trivial homomorphism $\xi=0$. 
\end{proof}

\begin{rem}\label{rem:brosore}
It follows from Corollary \ref{coro:tododentro} that any purely infinite algebra $R$ such that $[1_R]=0$ contains $L_2$ as a unital subalgebra. Hence by \cite{brosore}*{Theorem 4.1}, if $E$ is countable (resp. finite), then $L(E)$ embeds as a subalgebra (resp. a unital subalgebra) of $R$,  independently of whether $L(E)$ is simple or not. 
\end{rem}

\begin{coro}\label{coro:sospe1} Let $E$ be a countable graph with finite $E^0$. Assume that $K_0(L(E))$ is finite and let $d_1,\dots,d_n$, $d_i\backslash d_{i+1}$ be its invariant factors. Let $j:\aha\to kk$ be canonical functor (\cite{kkwt}). Then there is an algebra homomorphism $\psi:L(E)\to M_\infty(\bigoplus_{i=1}^nL_{d_i+1})$ such that $j(\psi)$ is an isomorphism in $kk$. If moreover $L(E)$ is purely infinite simple then there is an algebra homomorphism $\phi:\bigoplus_{i=1}^nL_{d_i+1}\to M_\infty L(E)$
such that $\iota^{-1}j(\phi)$ and $\iota^{-1}j(\psi)$ are inverse isomorphisms in $kk$. If $E$ is finite then the same holds with $L(E)$ substituted for $M_\infty(L(E))$. 
\end{coro}
\begin{proof}
Assume that $E$ is countable with finite $E^0$. By part ii) of Theorem \ref{thm:k0liftr}, for each $1\le i\le n$, there is a homomorphism $\psi_{i}:L(E)\to M_\infty L_{d_i+1}$ such that $K_0(\psi_i)$ is the projection from $K_0(L(E))=\bigoplus_{j=1}^n\Z/d_j$  onto the copy of $\Z/d_i$. 
The map 
\[
\psi=(\psi_1,\dots,\psi_{n}):L(E)\to M_\infty(\bigoplus_{i=1}^bL_{d_i+1})
\] 
then induces an isomorphism in $K_0$. In view of \cite{dwkk}*{Lemma 7.2} and of the fact that, since $K_0(L(E))$ is finite, $\ker(I-A_E^t)=0$, this implies that $K_1(\psi)$ is an isomorphism too. Hence $j(\psi)$ is an isomorphism by \cite{dwkk}*{Proposition 5.10}. Assume furthermore that $L(E)$ is purely infinite simple. Consider the graph 
\[
F=\coprod_{i=1}^n\cR_{d_i+1}.
\] 
Then $L(F)=\bigoplus_{j=1}^nL_{d_{j+1}}$. The homomorphism $\phi$ of the corollary is obtained by applying Theorem \ref{thm:k0liftr} to $\xi=K_0(\psi)^{-1}\iota:K_0(L(F))\to K_0(L(E))$. This proves the first assertion of the corollary; the second, for finite $E$, is proved similarly, using part iii) of Theorem \ref{thm:k0liftr}. 
\end{proof}

Let $E$ be a finite graph; if $X\subset L(E)$, write $\mspan(X)$ for the subspace generated by $X$. In the following proposition and elsewhere we consider the following ``diagonal" subalgebra of $L(E)$
\[
DL(E)=\mspan(\sink(E)\cup\{ee^*: e\in E^1\})\subset L(E).
\]

Proposition \ref{prop:k0agreeepsi} below will be needed in the next section.

\begin{prop}\label{prop:k0agreeepsi}
Let $E$ and $R$ be as in part iii) of Theorem \ref{thm:k0liftr}. Assume that $L(E)$ is simple and let $\phi, \psi : L(E) \to R$ be nonzero algebra homomorphisms such that $K_0(\phi) = K_0 (\psi)$. Then there exists an algebra homomorphism $\psi':L(E) \to R$ such that $j(\psi) = j(\psi')$ in $kk$ and $\psi'_{|DL(E)}=\phi_{|DL(E)}$.
\end{prop}
\begin{proof}
First assume that $ \phi(1) = \psi(1)=p$. For each $e\in E^1$ and each $v\in\sink(E)$ choose MvN equivalences $(x_e,y_e):\phi(ee^\ast)\sim \psi(e e^\ast)$ and $(x_v,y_v):\phi(v)\sim \psi(v)$.  Define $x = \sum_{e\in E^1} x_e+\sum_{v\in\sink(E)}x_v$ and $y = \sum_{e\in E^1} y_e+\sum_{v\in\sink(E)}y_v$. Then $x,y\in p R p$ and $x y = p= y x$. Hence $\psi':L(E)\to R$, $\psi'(a) = x \psi(a) y$ satisfies $ \psi'_{|DL(E)} = \phi_{|DL(E)}$. Moreover $j(\psi)=j(\psi')$ by \cite{dwkk}*{Lemma 2.3}.
Next assume that $\phi(1) \neq \psi(1)$ and that none of them is equal to $1$. Then by Corollary \ref{coro:equiv}, there is an element $u \in U(R)$ such that $u \phi (1) u^{-1} = \psi(1)$. Hence we can replace $\psi$ by $a\mapsto u \psi(a) u^{-1}$ and we are in the above case.
Finally, if $\phi(1) \neq \psi(1)$ and one of them, say $\psi(1)$, is $1$,  we can replace $\phi $ by a unital homomorphism by Theorem \ref{thm:k0liftr} and we are again in the first case.
\end{proof}

\section{Lifting \topdf{$K$}{K}-theory maps to algebra maps: \topdf{$K_0$}{K0} and \topdf{$K_1$}{K1}}\label{sec:k0k1lift}

Let $E$ be a finite graph; below we will give a right inverse of the surjective map
\begin{equation}\label{map:elonto}
\partial: K_1(L(E))\onto \ker(I-A_E^t).
\end{equation}
Observe that the analogue of the map \eqref{map:elonto} in the $C^*$-algebra setting is an isomorphism; an explicit formula for its inverse was given by R\o rdam
in \cite{ror}*{page 33} in the case when $E$ is regular. We shall show that in the purely algebraic case considered here, the same formula gives a right inverse of \eqref{map:elonto}, even for singular $E$.

Let $I-B^t_E$ be as in \cite{dwkk}*{Remark 5.7}. Let 
\[
s^*:\Z^{E^0}\to \Z^{(E^1)\coprod\sink(E)}, s^*(\chi_{v})=\left\{\begin{matrix}\sum_{s(e)=v}\chi_e& v\in\reg(E)\\ \chi_v& v\in\sink(E)\end{matrix}\right.
\] 
By \cite{abc}*{formula 4.1}, we have a commutative diagram
\[
\xymatrix{\Z^{E^1}\ar[r]^{I-B^t_E}&\Z^{E^1\coprod\sink(E)}\\
          \Z^{\reg(E)}\ar[u]^{s^*}\ar[r]_{I-A_E^t}&\Z^{E^0}\ar[u]^{s^*}
}
\]
In particular, $s^*$ maps $\ker(I-A^t_E)\to \ker(I-B^t_E)$. Furthermore it is an isomorphism by the dual of \cite{abc}*{Lemma 4.3}. Let $ x = (x_v) \in \ker(I-A_E^t) \subseteq \Z^{\reg(E)}$. Set $y=s^*(x)\in \ker(I-B_E^t)$. Let 
\begin{equation}\label{elS}
S=\{(e,j): y_e\ne 0, 1\le j\le |y_e|\}
\end{equation}
Consider the diagonal matrix $V=V(x)\in M_S(L(E))$, 
\[
V_{(e,j),(e,j)}=\left\{\begin{matrix} e& \text{ if } y_e>0\\ e^* & \text{ if } y_e<0\end{matrix}\right.
\]
Let $p=1-VV^*$, $q=1-V^*V$. Observe that $p,q\in M_S(DL(E))$. Moreover, for $\Alpha=E^1\coprod\sink(E)$, $DL(E)\cong \ell^{\Lambda}$ and we may regard 
$p=(p_{\alpha})$ and $q=(q_{\alpha})$ as $\Alpha$-tuples of diagonal matrices in $M_S$ whose entries are in $\{0,1\}$. One checks, using that $y\in\ker(I-B_E^t)$, that for each $\alpha\in\Alpha$, $p_{\alpha}$ and $q_{\alpha}$ have the same number of nonzero coefficients. Hence we may choose for each $\alpha$ a matrix 
$W_{\alpha}\in p_{\alpha}M_Sq_{\alpha}$ with coefficients in $\{0,1\}$ such that $W_{\alpha} W_{\alpha}^t=p_{\alpha}$ and $W_{\alpha}^tW_{\alpha}=q_{\alpha}$. Further, we may even require that
\begin{equation}\label{condiW}
(W_{\alpha})_{(e,i),(f,j)}=1\Rightarrow (p_{\alpha})_{(e,i),(e,i)}=(q_{\alpha})_{(f,j),(f,j)}=1.
\end{equation}
We shall use \eqref{condiW} in the proof of Lemma \ref{lem:aux} below. 
 Let $W=W(x)\in M_S(DL(E))$ be the matrix corresponding to $(W_\alpha)$; then 
\begin{equation}\label{vw}
WW^*=1-VV^*,\quad W^*W=1-V^*V \text{ and }W^*V=V^*W=0.
\end{equation} 
Put 
\begin{equation}\label{elU}
U(x)=V(x)+W(x).
\end{equation} 
It follows from \eqref{vw} that $U(x)U(x)^*=U(x)^*U(x)=1$.

\begin{prop}\label{prop:section}
Let $x\in \ker(I-A_E^t)$, $[U(x)]\in K_1(L(E))$ the class of the element \eqref{elU} and $\partial : K_1 (L(E))  \to \ker(I-A_E^t)$ as in \eqref{map:elonto}.
Then $\partial(U(x))=-x$.
\end{prop}
\begin{proof}
We keep the notation of the paragraph preceding the proposition. Let $C(E)$ be the Cohn path algebra; consider the subalgebra
\[
DC(E)=\mspan(\{q_v:v\in\reg(E)\}\cup\sink(E)\cup\{ee^*:e\in E^1\}\subset C(E).
\]
Consider the diagonal matrix $\hat{V}$ defined by the same prescription as $V$ but regarded now as an element of $M_S(C(E))$. Let $\hat{W}\in M_S(DC(E))$ be the image of $W$ under the map induced by the obvious inclusion $DL(E)\subset DC(E)$; put $\hat{U}=\hat{V}+\hat{W}$. 
Consider the matrix
$$ h = \begin{bmatrix}
   2 \hat{U} - \hat{U}\hat{U^\ast}\hat{U} & \hat{U}\hat{U^\ast}-1   \\           
    1-\hat{U^\ast}\hat{U} & \hat{U^\ast}
  \end{bmatrix}\in M_{S\coprod S}(C(E)). $$
  
By \cite{www}*{Section 2.4} (see also \cite{robook}*{Definition 9.1.3}), $h$ is invertible and  
$$\partial ( [U]) = [ h 1_{S} h^{-1}]-[1_{S}]. $$
Here $1_{S}$ is the $S\times S$ identity matrix, located in the upper left corner. 

One checks that $\hat{U}=\hat{U}\hat{U^\ast}\hat{U}$, and that 
\begin{equation}\label{deltU}
\partial ([U]) = [1- \hat{U^\ast}(x_i)\hat{U}]-[1 - \hat{U}\hat{U^\ast}]\in K_0(\bigoplus_{v\in \reg(E)}\ell q_v)\cong\Z^{\reg(E)}.
\end{equation} 
One checks, using \eqref{deltU} and the fact that $x\in \ker (I-A_E^t)$, that $$\partial([U])=-[\sum_{v \in E^0} x_v q_v].$$
This finishes the proof.
\end{proof}
 In principle, the assignment $\ker(I-A_E^t)\to K_1(L(E))$, $[x]\mapsto [U(x)]$ is just a set theoretic map. A group homomorphism with similar properties is obtained as follows. Choose a basis $\fB=\{x_i\}$ of the free abelian group $\ker(I-A_E^t)$; let
\begin{equation}\label{map:gamma}
\gamma=\gamma_{\fB}:\ker(I-A_E^t)\to K_1(L(E)), \quad \gamma(\sum_i n_ix_i)=\sum_in_i[U(x_i)].
\end{equation}

Let $E$ be a finite graph such that $L(E)$ is purely infinite simple. Then $\sink(E)=\emptyset$, by \cite{libro}*{Lemma 3.1.10 and Theorem 3.1.10}. Let $\phi : L(E) \to R$ be a unital algebra homomorphism with $R$ purely infinite simple. Set 
\begin{equation}\label{rphi}
R_\phi = \{ x \in R \ : \phi (e e^\ast)x=x\phi(ee^\ast),\quad \  \text{ for all } e \in E^1\}.
\end{equation}
Note that 
\[
R_\phi = \oplus_{e \in E^1} \phi (e e^\ast) R \phi (e e^\ast).
\] 
Because $L(E)$ is simple, $\phi(\alpha)\ne 0$ $(\alpha\in E^1)$, whence each of the inclusions $\phi(\alpha\alpha^*)R\phi(\alpha\alpha^*)\subset R$ induces an isomorphism in $K_1$. Hence the direct sum $R_\phi\subset R^{E^1}$ of those inclusions induces an isomorphism 
\begin{equation}\label{map:adjoint}
K_1(R_\phi) = \bigoplus_{e \in E^1} K_1( \phi (e e^\ast) R \phi (e e^\ast))\iso K_1(R)^{E^1}.
\end{equation}
Let $\iota:K_1(R_\phi)\to K_1(R)$ be the map induced by the inclusion $R_\phi\subset R$. Consider the bilinear map
\begin{equation}\label{pairing}
\langle \cdot , \cdot \rangle: \Z^{E^1}\times K_1(R_\phi)\to K_1(R), \quad \langle x,y\rangle=\sum_i 
x_i\iota(y_i).
\end{equation}
Observe that $\langle \cdot , \cdot \rangle$ is a perfect pairing; indeed the adjoint homomorphism $K_1(R_\phi)\to K_1(R)^{E^1}$ is the isomorphism 
\eqref{map:adjoint}. 

\begin{lema}\label{lem:aux} (cf.\cite{ror}*{Lemma 3.5}.)
Let $E$ be a finite graph such that $L(E)$ is purely infinite simple, $R$ a purely infinite simple unital algebra, and $\phi$ and $\psi : L(E) \to R$ unital homomorphisms. Assume that $\phi$ and $\psi$ agree on $DL(E)$. Let 
\[
u=\sum_{\alpha \in E^1} \psi(\alpha) \phi(\alpha^\ast)\in R_\phi= R_\psi.
\]
Then
\begin{equation}\label{efsdfas}
 K_1(\psi) (\gamma(x ))= \langle x , [u] \rangle + K_1(\phi)   (\gamma(x )) \text{ for all } x \in \ker(I-A_E^t).
\end{equation}
\end{lema}

\begin{proof}
Observe that $\psi(e)\phi(e^*)\in U(\phi(e)R\phi(e^*))$ $(e\in E^1)$, whence $u\in U(R_\phi)$. 
Let $\{\chi_e:e\in I\}$ be the canonical basis of $\in\Z^I$. One checks that 
\begin{equation}\label{paircano}
 \langle \chi_e , [u] \rangle = \psi(e) \phi(e^\ast) +1 -\phi(ee^\ast).
\end{equation}
To prove the lemma, we may assume that $x$ is an element of the basis $\fB$ of $\ker(I-A_E^t)$ used in \eqref{map:gamma} to define $\gamma$ . Then taking 
\eqref{paircano} into account and adopting the notations and conventions used in the definition of $U(x)$, one computes that the right hand side of equation \eqref{efsdfas} is
\begin{equation}\label{rhs:aux}
\sum_{y_e > 0} y_e [\psi(e) \phi(e^\ast) +1 -\phi(e e^\ast)] + [\phi ( U(x))] - \sum_{y_e < 0} y_e [\phi(e) \psi(e^\ast) +1 -\phi(e e^\ast)] .
\end{equation}
Let $S$ be as in \eqref{elS}. Consider the diagonal matrices $P,Q\in M_SL(E)$ with diagonal entries as follows
\begin{gather*}
P_{(e,j),(e,j)}=\left\{\begin{matrix}\psi(e)\phi(e^*)+1-\phi(ee^*)& \text{ if }y_e>0\\ 1 & \text{ if } y_e<0\end{matrix}\right.\\
Q_{(e,j),(e,j)}=\left\{\begin{matrix}1 & \text{ if }y_e>0\\ \phi(e)\psi(e^*)+1-\phi(ee^*)& \text{ if } y_e<0\end{matrix}\right.
\end{gather*}
Observe that \eqref{rhs:aux} is $[P\phi(U(x))Q]$. Hence it suffices to show that $K_1(\psi)(U(x))=[P\phi(U(x))Q]$; we shall show that in fact $\psi(U(x))=P\phi(U(x))Q$. Recall that $U(x)=V(x)+W(x)$. It is immediate from the definition of $V(x)$ that $\psi(V(x))=P\phi(V(x))Q$. Hence since $W$ has coefficients in $DL(E)$, it only remains to show that $\phi(W(x))=P\phi(W(x))Q$. A tedious but straightforward calculation, using \eqref{condiW} shows that 
\[
\phi(W(x)_{\alpha})_{(e,i),(f,j)}=(P\phi(W(x)_\alpha)Q)_{(e,i),(f,j)}\qquad\forall (e,i),(f,j)\in S,\quad \alpha\in\Lambda.
\]
This completes the proof.
\end{proof}

\begin{rem}\label{rem:xi0detxi1} 
Recall that if $L(E)$ is unital, we have an exact sequence
\[
0\to K_0(L(E))\otimes K_1(\ell)\to K_1(L(E))\to \ker(I-A_E^t)\to 0.
\]
It follows from \cite{dwkk}*{Lemma 7.2} that if $R\in\aha$ is $K_1$-regular and $\xi\in kk(L(E),R)$, then $K_1(\xi)$ restricts to the composite of $K_0(\xi)\otimes \id$ with the cup product $K_0(R)\otimes K_1(\ell)\to K_1(R)$.
\end{rem} 

\begin{teo}\label{thm:ror}
Let $E$ be a finite graph and $S$ an algebra. Assume that $L(E)$ is simple and that $S$ is unital, purely infinite simple and $K_1$-regular. Let $\xi_0 : K_0 (L(E)) \to K_0(S)$ and $\xi_1 : \ker(I-A_E^t) \to K_1(S)$ be group homomorphisms. 
Then there exists a nonzero algebra homomorphism $\phi : L(E) \to  S$ such that $K_0(\phi) = \xi_0$ and such that $K_1(\phi)\gamma= \xi_1$. If moreover $\xi_0$ is unital then we can choose $\phi$ to be a unital homomorphism $L(E)\to S$.
\end{teo}

\begin{proof}
By Theorem \ref{thm:k0liftr}, there exists a nonzero algebra homomorphism \goodbreak
$ \phi_0 : L(E) \to  S$ such that $K_0(\phi_0) = \xi_0$, and if $\xi_0$ is unital then we may choose $\phi_0$ unital. If $L(E)$ is not purely infinite, then by Remark \ref{rem:simpnopis}, $L(E)\cong M_n$ for some $1\le n<\infty$. Hence
$\ker(I-A_E^t)=0$ and $K_1(L(E))=K_0(L(E))\otimes U(\ell)$. Assume that $L(E)$ is purely infinite simple. Let $R = \phi_0(1) S \phi_0(1)$ and let $\bar{\phi}_0 : L(E) \to R$ be the corestriction of $\phi_0$ and $\inc : R \to  S$ the inclusion. Since $\ker(I-A_E^t)$ is a direct summand of $\Z^{\reg(E)}$ and $\langle\cdot,\cdot\rangle$ is a perfect pairing, there exists $\theta\in K_1(R_{\bar{\phi}_0})$ such that
$$ \langle - , \theta \rangle = K_1(\inc)^{-1} \xi_1 - K_1(\bar{\phi}_0) \gamma. $$
Because $R_{\bar{\phi}_0}$ is a direct sum of purely infinite simple algebras, by Theorem \ref{thm:kpis} there exists $g \in U(R_{\bar{\phi}_0})$ such that $[g]= \theta$. Define $\bar{\phi} : L(E) \to R$ by setting $\bar{\phi}_{|E^0}=(\bar{\phi_0})_{|E^0}$,  $\bar{\phi}(e)=g\bar{\phi_0}(e)$, 
$\bar{\phi}(e^\ast)=\bar{\phi_0}(e^\ast) g^{-1}$. Observe that $\bar{\phi}$ and $\bar{\phi}_0$ agree on $DL(E)$; in particular, $\bar{\phi}$ is unital.  Hence by Lemma \ref{lem:aux}, we have
$$ K_1(\bar{\phi}) \gamma = K_1(\bar{\phi}_0 ) \gamma + \langle - , [u]\rangle. $$
But it follows from the formula defining $u$ in Lemma \ref{lem:aux} and the definition of $\bar{\phi}$ that $u=g$. Hence
$$ K_1(\bar{\phi}) \gamma = K_1(\inc)^{-1} \xi_1. $$
Set $\phi=\inc\bar{\phi}$; then $ K_1(\phi)\gamma= \xi_1$. Further, $K_0(\phi) = K_0(\phi_0) = \xi_0$ because $\phi$ and $\phi_0$ agree on $E^0$. 
It is clear by construction that if $\phi_0$ is unital homomorphism, then $\phi$ is also unital. 
\end{proof}

\section{Lifting \topdf{$kk$}{kk}-maps to algebra maps}\label{sec:kklift}

Let $\phi,\psi:A\to B$ be algebra homomorphisms. Put
\[
C_{\phi,\psi}=\{(a,f)\in A\oplus B[t]:f(0)=\phi(a), f(1)=\psi(a)\}.
\]
Let $\pi:C_{\phi,\psi}\to A$, $\pi(a,f)=a$; we have an algebra extension
\begin{equation}\label{seq:cylinder}
\Omega B\to C_{\phi,\psi}\overset{\pi}{\lra} A. 
\end{equation}
\begin{lem}\label{lem:triaresta}
Let $j:\aha\to kk$ be the canonical functor. The sequence \eqref{seq:cylinder} induces the following distinguished triangle in $kk$
\[
\xymatrix{j(\Omega B)\ar[r]& j(C_{\phi,\psi})\ar[r]^{j(\pi)}& j(A)\ar[r]^{j(\phi)-j(\psi)}&j(B).}
\] 
\end{lem}
\begin{proof}
By definition of $C_{\phi,\psi}$, we have a map of extensions 
\begin{equation}\label{map:cylseq}
\xymatrix{\Omega B\ar@{=}[d]\ar[r]&C_{\phi,\psi}\ar[r]^\pi\ar[d]&A\ar[d]^{(\phi,\psi)}\\
\Omega B\ar[r]& B[t]\ar[r]^{(\ev_0,\ev_1)}& B\oplus B}
\end{equation}
Let $\Delta:B\to B\oplus B$, $\Delta(b)=(b,b)$. One checks that the $kk$-triangle associated to the bottom row of \eqref{map:cylseq} is isomorphic to 
\[
\xymatrix{j(\Omega B)\ar[r]^0&j(B)\ar[r]^{j(\Delta)}&j(B\oplus B)\ar[r]^(.6){[\id,-\id]}&B.}
\] 
Let $\xi:j(A)\to j(B)$ be the boundary map in the triangle induced by \eqref{seq:cylinder}. It follows from \eqref{map:cylseq} that there is a commutative diagram 
\[
\xymatrix{j(A)\ar[r]^\xi\ar[d]^{(j(\phi),j(\psi))}&j(B)\ar@{=}[d]\\
        j(B)\oplus j(B)\ar[r]_{[\id,-\id]}& j(B).}
\]
Hence $\xi=j(\phi)-j(\psi)$.
\end{proof}
Let $R$ be a unital, purely infinite simple algebra, let $E$ be a finite graph such that $L(E)$ is simple and let $\phi,\psi:L(E)\to R$ be nonzero algebra homomorphisms which agree on $DL(E)$. Let $R_\phi$ be as
in \eqref{rphi}. Put $p=\phi(1)$ and let
\[
B=pRp.
\]
By corestriction, we may consider $\phi$ and $\psi$ as homomorphisms $L(E)\to B$. Let
\[
C=\{f\in B[t]\mid (\exists a\in L(E))\ \ \phi(a)=f(0),\ \ \psi(a)=f(1)\}.
\]
Since $L(E)$ is simple, the map 
\[
C_{\phi,\psi}\to C,\quad (a,f)\mapsto f
\]
is an isomorphism. We shall identify $C=C_{\phi,\psi}$. Assume that $R$ is $K_1$-regular. Then $B$ is $K_1$-regular also, whence $K_0(\Omega B)=KV_1(B)=K_1(B)$. Hence the extension \eqref{seq:cylinder} induces an exact sequence
\begin{equation}\label{seq:kcyl}
\xymatrix{K_1(B)\ar[r]^{\partial'}&K_0(C)\ar[r]^\pi& K_0(L(E))\ar[r]^{\phi-\psi}&K_0(B)}
\end{equation}

The following two lemmas adapt \cite{ror}*{Lemmas 3.2 and 3.3} to the purely algebraic case. 

\begin{lem}\label{lem:3.2}
Let $u$ be as in Lemma \ref{lem:aux}, $\partial'$ as in \eqref{seq:kcyl} and $\langle \cdot,\cdot\rangle$ as in \eqref{pairing}. Let $\sigma\in K_0(C)^{E^1}$, $\sigma_e=[\phi(ee^*)]$. Then for every $x\in \Z^{E^1}$ we have
\[
\langle x,[u]\rangle=-\langle(I-A^t_E)x,\sigma\rangle
\]
\end{lem}
\begin{proof} Let $u_e=u\phi(ee^*)+1-\phi(ee^*)$ ($e\in E^1$). By Whitehead's lemma there is $U_e(t)\in \GL(B[t])$ such that $U_e(0)=1$ and
$U_e(1)=\diag(u_e,u_e^{-1})$. Set $V_e(t)=U_e(t)\diag(\phi(e),0)$, $W_e(t)=\diag(\phi(e^*),0)U_e(t)^{-1}$. Now proceed as in the proof of 
\cite{ror}*{Lemma 3.2}, substituting $U_e(t)^{-1}$ and $W_e(t)$ for $U_e(t)^*$ and $V_e(t)^*$.
\end{proof}

\begin{lem}\label{lem:3.3}
Let $\lambda:R_\phi\to R_\phi$, $\lambda(a)=\sum_{e\in E^1}\phi(e)a\phi(e^*)$. If $j(\phi)=j(\psi)\in kk(L(E),B)$ then there is $\nu\in U(R_\phi)$
such that $[u]=[\nu^{-1}\lambda(\nu)]\in K_1(R_\phi)$.
\end{lem}
\begin{proof} The proof is the same as that of \cite{ror}*{Lemma 3.3}.
\end{proof}

Let $S$ be an algebra, $E$ a finite graph, and $\phi,\psi:L(E)\to S$ algebra homomorphisms. We say that $\phi$ and $\psi$ are \emph{$1$-step $\ad$-homotopic} if either
\begin{itemize}
\item[a)] there is an MvN equivalence $(u,u'):\psi(1)\sim\phi(1)$ such that $\ad(u,u')\phi=\psi$,
\goodbreak
or
\goodbreak
\item[b)] $\phi$ and $\psi$ agree on $DL(E)$ and for $B=\phi(1)S\phi(1)$ there is $U(t)\in \GL(B_{\phi}[t])$ such that $U(0)=1$ and $\phi_{i+1}(e)=U(1)\phi(e)$, $\psi(e^*)=\phi_i(e^*)U(1)^{-1}$.
\end{itemize}
We say that $\phi$ and $\psi$ are \emph{$n$-step $\ad$-homotopic} if there is a sequence of algebra homomorphisms $\phi_i:L(E)\to S$, $1\le i\le n$, such that $\phi_1=\phi$, $\phi_n=\psi$, and $\phi_{i}$ and $\phi_{i+1}$ are $1$-step $\ad$-homotopic for $1\le i\le n-1$. Two unital homomorphisms $\phi$ and $\psi$ are \emph{$n$-step unitally $\ad$-homotopic} if they are $n$-$\ad$-homotopic and the $\phi_i$ can be chosen to be unital for all $1\le i\le n$. Call $\phi$ and $\psi$ (unitally) \emph{$\ad$-homotopic} if they are $n$-step (unitally) $\ad$-homotopic for some $n$.
  
\begin{rem}\label{rem:adhomo}
Observe that if in a) above $\phi$ and $\psi$ are unital, then $u\in U(S)$, so that $\phi$ and $\psi$ are conjugate in the usual, unital sense. Note also that in the situation b) above, $\phi$ and $\psi$ are homotopic. It follows that a unital homomorphism $\phi:L(E)\to L(E)$ is unitally $\ad$-homotopic to the identity if and only if it is homotopic to $\ad(u)$ for some $u\in U(L(E))$.  
\end{rem}	
		
\begin{thm}\label{thm:kkliftr} Let $E$ be a finite graph and $R$ a unital algebra. Assume that $L(E)$ and $R$ are purely infinite simple and that $R$ is $K_1$-regular. Then the canonical map 
\begin{equation}\label{map:kkliftr}
j : [L(E), R]_{M_2}\setminus\{0\} \to kk(L(E),R)
\end{equation}
is an isomorphism of monoids. In particular, $[L(E), R]_{M_2} \setminus\{0\}$ is the group completion of $[L(E),R]_{M_2}$. Moreover, we have the following:

\item[i)] If $\xi\in kk(L(E),R)$, then there is a nonzero algebra homomorphism $\phi:L(E)\to R$ such that $j(\phi)=\xi$. Moreover, $\phi$ may be chosen to be unital if $\xi$ is.

\item[ii)] Two nonzero (unital) algebra homomorphisms $\phi,\psi: L(E)\to R$ satisfy $j(\phi)=j(\psi)$ if and only if they are $M_2$-homotopic if and only if they are (unitally) $\ad$-homotopic if and only if they are $3$-step (unitally) $\ad$-homotopic. 
\end{thm}
\begin{proof} The map $[L(E),R]_{M_2} \to kk(L(E),R)$ is a monoid homomorphism by the same argument as in Proposition \ref{prop:homotopis}.

Let $\xi\in kk(L(E),R)$ and let $\gamma:\ker(I-A_E^t)\to K_1(L(E))$ be as in \eqref{map:gamma}. By Theorem \ref{thm:ror} there exists
a nonzero algebra homomorphism $\psi:L(E)\to  R$ such that $K_0(\xi)=K_0(\psi)$ and $K_1(\xi)\gamma=K_1(\psi)\gamma$. Let $B=\psi(1) R\psi(1)$, $\inc : B \to R $ the inclusion and $\bar{\psi}:L(E)\to B$ the corestriction of $\psi$. Then $j(\inc)$ is an isomorphism, and 
for $\eta=j(\inc)^{-1}\xi$ we have $\eta-j(\bar{\psi})\in kk(L(E),B)^2\cong \Ext_\Z^1(K_0(L(E)),K_1(B))$, by \cite{dwkk}*{Theorem 7.11}. To prove that the map of the theorem is surjective, it suffices to show that there exists $u\in U(R_\psi)$ such that for $\phi:L(E)\to B$, $\phi(e)=u\psi(e)$, $\phi(e^*)=\psi(e^*)u^{-1}$, we have $\eta-j(\bar{\psi})=j(\phi)-j(\bar{\psi})$. The argument of the proof of \cite{ror}*{Theorem 3.1} shows this. Next we show that \eqref{map:kkliftr} is injective, and that the different notions of homotopy agree. It follows from Lemma \ref{lem:adhomo}, \cite{dwkk}*{Lemma 2.3} and the definition of $\ad$-homotopy that $\ad$-homomotopic homomorphisms $L(E)\to R$ are $M_2$-homotopic, and from the universal property of $kk$ that $j$ sends homotopic maps to equal maps. Conversely, let $\phi,\psi:L(E)\to  R$ be algebra homomorphisms such that $j(\phi)=j(\psi)$. Then $K_0(\phi)=K_0(\psi)$, whence there exist
for each $e\in E^1$ elements $u_e\in \phi(ee^*) R\psi(ee^*)$ and $u'_e\in \psi(ee^*) R\phi(ee^*)$ such that $u_eu'_e=\phi(ee^*)$ and $u'_eu_e=\psi(ee^*)$. Thus $u=\sum_{e\in E^1}u_e\in \phi(1)R\psi(1)$, $u'=\sum_{e\in E^1}u'_e\in\psi(1) R\phi(1) $, and $\psi'=\ad(u,u')\psi$ agrees with $\phi$ on $DL(E)$. Hence upon spending a $1$-step $\ad$-homotopy from $\psi$ to $\psi'$ if necessary, we may assume that $\phi$ and $\psi$ agree on $DL(E)$. Let  $B=\phi(1) R\phi(1)$ and let $u\in B_\phi$ be as in Lemma \ref{lem:aux}; we have
\begin{equation}\label{phiu}
\psi(e)=u\phi(e), \quad \psi(e^*)=\phi(e^*)u^{-1}.
\end{equation}
Observe that, because $R$ is purely infinite and $K_1$-regular, the same is true of $B$. By Lemma \ref{lem:3.3} and $K_1$-regularity of $B$, there are $\nu\in\GL(B_\phi)$ and 
$U(t)\in \GL(B[t])$ such that $U(0)=1$ and $U(1)=\nu^{-1}\lambda(\nu)u^{-1}$. Hence, upon using a second $1$-step $\ad$-homotopy, we may assume that $u=\nu^{-1}\lambda(\nu)$. A calculation shows that $\psi=\ad(\nu)\phi$. Thus a third $1$-step $\ad$-homotopy concludes the proof of the nonunital part of the theorem. If $\xi$ is unital, then by Theorem \ref{thm:ror} there is a unital algebra homomorphism $\psi:L(E)\to R$ such that $K_0(\xi)=K_0(\psi)$ and $K_1(\xi)\gamma=K_1(\psi)\gamma$. The argument used above to prove the surjectivity of \eqref{map:kkliftr} subsituting $\xi$ for $\eta$, shows that there is a unital algebra homomorphism $\phi:L(E)\to R$ such that $j(\phi)=\xi$. Finally the same argument used above for nonunital homomorphisms shows that two unital homomorphisms $L(E)\to R$ go to the same element in $kk$ if and only if they are unitally $3$-step $\ad$-homotopic.  
\end{proof}

\begin{rem}\label{rem:gagp}
 By Lemma \ref{lem:adhomo}, we have that if $R$ and $L(E)$ are as in Theorem \ref{thm:kkliftr}, then $[L(E), M_\infty R]$ is an abelian monoid, with operation induced by the map \eqref{map:boxplus}, and the canonical homomorphism  $[L(E), M_\infty R] \setminus \{ 0 \}\to kk(L(E),R)$ is an isomorphism of monoids.
\end{rem}

\section{Homotopy classification theorem}\label{sec:main}
\begin{thm}\label{thm:main2}
Let $E$ and $F$ be finite graphs such that $L(E)$ and $L(F)$ are purely infinite simple. 
Let $\xi_0:K_0(L(E))\iso K_0(L(F))$ be an isomorphism. Then there exist nonzero algebra homomorphisms $\phi:L(E)\to L(F)$ and $\psi:L(F)\to L(E)$ such that $K_0(\phi)=\xi_0$, $K_1(\psi)=\xi_0^{-1}$, $\psi\phi\simh_{M_2}\id_{L(E)}$ and $\phi\psi\simh_{M_2}\id_{L(F)}$. If $\xi_0$ is unital then we may choose $\phi$ and $\psi$ to be unital homomorphisms such that $\phi\psi$ and $\psi\phi$ are homotopic to the respective identity maps.
\end{thm}
\begin{proof}
Because $\ker(I-A_E^t)$ and $\ker(I-A_F^t)$ are isomorphic to the quotients of $K_0(L(E))$ and $K_0(L(F))$ modulo torsion, the assumed isomorphism $\xi_0$ induces an isomorphism $\xi_1 : \ker(I-A_E^t) \iso \ker(I-A_F^t)$. By \cite{dwkk}*{Corollary 7.19}, there exists
$\xi\in kk(L(E),L(F))$ such that for the injective homomorphism $\gamma_F:\ker(I-A_F^t)\to K_1(L(F))$ of \eqref{map:gamma}, we have $K_0(\xi)=\xi_0$ and $K_1(\xi)\gamma_E=\gamma_F\xi_1$. Hence $\xi\in kk(L(E),L(F))$ is an isomorphism by \cite{dwkk}*{Proposition 5.10}. By Theorem \ref{thm:kkliftr} there are algebra homomorphisms $\phi:L(E)\to L(F)$ and $\psi:L(F)\to L(E)$ such that $j(\phi)=\xi$ and $j(\psi)=\xi^{-1}$, which may be chosen unital if $\xi_0$ is. Again by Theorem \ref{thm:kkliftr}, $\phi\psi$ and $\psi\phi$ are $M_2$-homotopic to the respective identity maps. If moreover $\phi$ and $\psi$ are unital, then by Theorem \ref{thm:kkliftr}, $\phi\psi$ and $\psi\phi$ are unitally $\ad$-homotopic to identity maps.  Hence by Remark \ref{rem:adhomo} there are $u\in U(L(E))$ and $v\in U(L(F)$ such that $\ad(v)\phi\psi$ and $\psi\phi\ad(u)$ are homotopic to identity maps. Hence $\psi$ is a homotopy equivalence. Upon replacing $\phi$ by the homotopy inverse of $\psi$, we get the last statement of the theorem.    
\end{proof}

Recall from \cite{black}*{Chapter III, Section 6.2} that a \emph{scaled ordered group} is an ordered group together with a choice of order unit. If $R$ is a unital algebra, then $K_0(R)$
has a natural structure of scaled ordered group whose positive cone is the image of $\cV(R)$ and whose order unit is $[1_R]$. 

We say that two unital algebras $R$ and $S$ are \emph{unitally homotopy equivalent} if there are unital homomorphisms $\phi:R\to S$ and $\psi:S\to R$ such that $\psi\phi$ and $\phi\psi$ are homotopic to the respective identity maps.

\begin{coro}\label{coro:main2} Let $E$ and $F$ be finite graphs such that $L(E)$ and $L(F)$ are simple. Assume that $K_0(L(E))$ and $K_0(L(F))$ are isomorphic as scaled ordered groups. Then either
\item[i)] there is $1\le n$ such that $L(E)\cong L(F)\cong M_n$
\goodbreak
or
\goodbreak
\item[ii)] $L(E)$ and $L(F)$ are purely infinite and unitally homotopy equivalent.  
\end{coro}
\begin{proof} By Remark \ref{rem:simpnopis} if $L(E)$ is simple but not purely infinite, then there is $n\ge 1$ such that 
$L(E)\cong M_n$. In this case $K_0(L(E))\cong \Z$ with the usual order and $[1_{L(E)}]$ corresponds to $n$. On the other hand if $R$ is a purely infinite simple unital algebra, then every element of $K_0(R)$ is nonnegative, by Theorem \ref{thm:kpis}. The proof is concluded using Theorem \ref{thm:main2} and observing that the identity is the only automorphism of $\Z$ as an ordered group. 
\end{proof}

\section{Algebra extensions}\label{sec:ext}
Let $R$ be an algebra. For $x\in R^\N$, let $\supp (x)=\{n\in\N:x_n\ne 0\}$. For a matrix $a\in R^{\N\times\N}$ and $i\in\N$, put $a_{i,*}$ and $a_{*,i}$ for the $i^{th}$ row and column, and set
\begin{gather*}
\Im (a)=\{a_{i,j}:i,j\in\N\}\subset R,\\
N(a)=\sup\{\#\supp(a_{i,*}), \#\supp(a_{*,i}):i\in\N\}. 
\end{gather*} 
Consider the algebras 
\begin{gather*}
\Gamma R=\{a\in R^{\N\times\N}: \text{each row and column of $a$ is finitely supported}\}, \\
\Gamma' R=\{a\in\Gamma R: \#\Im(a)<\infty \text{ and } N(a)<\infty\},\\
\Sigma R=\Gamma R/M_\infty R,\quad\Sigma'R=\Gamma'R/M_\infty R.  
\end{gather*}
The algebras $\Sigma R$ and $\Sigma'R$ are the \emph{Wagoner} and \emph{Karoubi} suspensions.
\begin{prop}\label{prop:wagopis}
Let $R$ be either a division algebra or a purely infinite simple unital algebra. Then $\Sigma R$ and $\Sigma' R$ are purely infinite simple. 
\end{prop}
\begin{proof} 
It suffices to show that if $M\in \Gamma R\setminus M_\infty R$ then there are $A,B\in \Gamma' R$ such that $AMB=1$. The conditions defining 
$\Gamma'$ and $\Gamma$ imply that there are infinite, strictly increasing sequences $Y=\{y_1, y_2,\dots\}, N=\{N_1=1,N_2,\dots\}\subset\N$ such that for each $j$, $\emptyset\ne\supp(m_{*,y_j})\subset[N_j+1,N_{j+1}]$. Let $B_1$ be the matrix whose $n^{th}$ column is the canonical basis element $\epsilon_{y_n}$. The support of the $j^{th}$-column of the matrix $MB_1$ is contained in $[N_j+1,N_{j+1}]$. Choose, for each $j$, an element $x_j\in [N_{j+1},N_{j+1}]$ such that $(MB_1)_{x_j,j}\ne 0$. Let $A_1$ be the matrix whose $j^{th}$ row is the basis element $\epsilon_{x_j}$. The matrix $A_1MB_1$ is diagonal, and all its diagonal entries are nonzero. Hence by our hypothesis on $R$ there are diagonal matrices $A_2$ and $B_2$ such that $A_2A_1MB_1B_2=1$.  
\end{proof}

Recall from \cite{dwkk}*{Lemma 2.8 and the paragraph below} that when $R$ is unital, every extension of an algebra $A$ by $M_\infty R$ is classified by a homomorphism $A\to \Sigma R$. By \cite{dwkk}*{Lemma 2.5}, the sets $[A,\Sigma R]_{M_2}$ and
$[A,\Sigma'R]_{M_2}$ are abelian monoids with the sum induced by \eqref{map:boxplus}. Put
\[
\cExt(A,R)=[A,\Sigma R]_{M_2},\quad \cExt(A,R)_f=[A,\Sigma'R]_{M_2}.
\]
By definition, there is a canonical map $\cExt(A,R)_f\to\cExt(A,R)$; by \cite{dwkk}*{Remark 5.8} there is also a natural map $\cExt(A,R)\to kk_{-1}(A,R)$. 
\begin{thm}\label{thm:ext}
Let $R$ be either a division algebra or a $K_0$-regular purely infinite simple unital algebra and $E$ a finite graph such that $L(E)$ is simple. Then the canonical maps
\[
\cExt(L(E),R)_f\to\cExt(L(E),R)\to kk_{-1}(L(E),R) 
\]
are isomorphisms. Moreover every nonzero element of each of these groups represents the $M_2$-homotopy class a nontrivial extension of $L(E)$ by $M_\infty(R)$. 
\end{thm}
\begin{proof} 
Since $\ell$ is a field, $\Sigma$ and $\Sigma'$ are models for the suspension functor. By Proposition \ref{prop:wagopis}, $\Sigma R$ and $\Sigma ' R$ are purely infinite simple. Now apply Theorem \ref{thm:kkliftr} to prove the first assertion. The second assertion follows from Theorem \ref{thm:kkliftr} and \cite{dwkk}*{Lemma 2.8}.
\end{proof}

\begin{coro}\label{coro:ckext}[cf. \cite{ck}*{Theorem 5.3}]
For $E$ as in the theorem above, we have
\[
\cExt(L(E),\ell)=\coker(I-A_E).
\] 
\end{coro}
\begin{proof} Immediate from Theorem \ref{thm:ext} and the the fact that $KH^1(L(E))=\coker(I-A_E)$ \cite{dwkk}*{Formula 6.4}.  
\end{proof}

\begin{coro}\label{coro:cute}
Let $E$ and $R$ be as in Theorem \ref{thm:ext}. Then there is an exact sequence
\begin{multline*}
0\to\Ext^1_\Z(K_0(L(E)),K_0(R))\to \cExt(L(E),R)\to\\
\Hom_\Z(\ker(I-A_E^t),K_0(R))\oplus\Hom_\Z(K_0(L(E)),K_{-1}R)\to 0.
\end{multline*}
\end{coro}
\begin{proof} Immediate from Theorem \ref{thm:ext} and \cite{dwkk}*{Corollary 7.19}.
\end{proof}

\begin{ex}\label{ex:cute}
If $R$ is either $\ell$ or a purely infinite simple unital Leavitt path algebra, then $K_{-1}R=0$, so the exact sequence
of Corollary \ref{coro:cute} becomes
\[
0\to\Ext^1_\Z(K_0(L(E)),K_0(R))\to \cExt(L(E),R)\to
\Hom_\Z(\ker(I-A_E^t),K_0(R))\to 0.
\]
If furthermore $K_0(L(E))$ is torsion, then $\ker(I-A_E^t)=0$, and we get a canonical isomorphism
\[
\cExt(L(E),R)=\Ext^1_\Z(K_0(L(E)),K_0(R)).
\]
\end{ex}

\section{Maps into tensor products with \topdf{$L_2$}{L2}}\label{sec:tenso}

\begin{lem}\label{lem:homozero} Let $E$ be a graph and let $\phi:L(E)\to R$ be an algebra homomorphism. Then $\phi\simh 0\iff \phi=0$.
\end{lem}
\begin{proof} It suffices to show that if $H:L(E)\to R[t]$ satisfies $\ev_0H=0$, and $v\in E^0$, then $H(v)=0$. This follows from Lemma \ref{lem:idemzero}.
\end{proof}

A unital algebra $R$ is \emph{regular supercoherent} if for every $n\ge 0$, $R[t_1, \dots, t_n]$ is regular coherent in the sense of \cite{gersten}.

\begin{lem}\label{lem:kreg} 
Let $E$ be graph and $R$ a regular supercoherent algebra. Then $L(E)\otimes R$ is $K$-regular. In particular, $L(E) \otimes L(F)$ is $K$-regular for every finite graph $F$.
\end{lem}
\begin{proof} 

By definition, $R_n= R[t_1,\dots,t_n]$ is regular supercoherent for every $n\ge 0$. Hence by \cite{dwkk}*{Example 5.5} the canonical map $K_*(R_n\otimes L(E))\iso KH_*(R_n\otimes L(E))=KH_*(R_0\otimes L(E))$ is an isomorphism for every $n$, whence also $K_*(R_0\otimes L(E))\to K_*(R_n\otimes L(E))$ is an isomorphism for all $n$. The second assertion follows from the first, using \cite{libro}*{Lemma 6.4.16}.
\end{proof}

Let $R,S$ be unital algebras. Put
\[
[R,S]\supset [R,S]_1=\{[f]: f \text{ unital }\}.
\]

\begin{thm}\label{thm:mapl2}
Let $E$ be finite graph such that $L(E)$ is simple and $R$ a purely infinite simple regular supercoherent algebra. Then $[L(E),L_2]_1=[L(E),L_2]_{M_2}\setminus\{0\}$, $[L(E),R\otimes L_2]_1=[L(E),R\otimes L_2]_{M_2}$, and both sets have exactly one element each.  
\end{thm}

\begin{proof} By Remark \ref{rem:simpnopis}, Proposition \ref{prop:homotopis} and Theorem \ref{intro:kklift}, $[L(E),L_2]_{M_2}\setminus\{0\}$ has exactly one element, since $j(L_2)=0$ in $kk$; by Corollary \ref{coro:tododentro} this element is the class of a unital homomorphism. Next let $\phi,\psi:L(E)\to L_2$ be unital homomorphisms. If $L(E)$ is not purely infinite, then by Proposition \ref{prop:homotopis}, $\phi$ and $\psi$ are conjugate, and therefore homotopic, since by Corollary \ref{coro:kv1pis}, $ \pi_0(U(L_2)) = K_1(L_2) = 0$. If $L(E)$ is purely infinite, then by part iii) of Theorem \ref{thm:kkliftr}, $\phi$ and $\psi$ are $3$-step unitally ad-homotopic. Hence by Remark \ref{rem:adhomo} and the argument we have just used, $\phi\simh\psi$. Thus the assertions about homomorphisms $L(E)\to L_2$ are proved. It is well-known that the tensor product of a unital simple algebra with a unital central simple algebra is again simple. By \cite{apcrow}*{Theorem 4.2}, $L_2$ is central, so $R\otimes L_2$ is simple. Moreover, $R\otimes L_2$ is purely infinite by \cite{apgpsm}*{Theorem 7.9}. Hence using that $j(R\otimes L_2)=0$ in $kk$ and applying Lemmas  \ref{lem:homozero} and \ref{lem:kreg}, Proposition \ref{prop:homotopis} and Theorem \ref{thm:kkliftr}, we obtain
\[
[L(E),R\otimes L_2]_{M_2}\setminus\{0\} = kk(L(E),R\otimes L_2)=0.
\] 
By Corollary \ref{coro:tododentro} there is a unital homomorphism $\phi:L(E)\to L(F)\otimes L_2$. If $\psi$ is another, then $\phi\simh\psi$ by Lemma \ref{lem:adxy} and the argument above.   
\end{proof}

\begin{ex}\label{ex:otimesl2} 
Let $R$ be as in Theorem \ref{thm:mapl2}, let $d:L_2\to R \otimes L_2$, $a\mapsto 1\otimes a$ and let $\phi:L_2\to R\otimes L_2$ be any homomorphism. Setting $L(E)=L_2$ in Theorem \ref{thm:mapl2} we get that if $\phi$ is nonzero then it is $M_2$-homotopic to $d$ and that if $\phi$ is unital then it is homotopic to $d$.
\end{ex}

\end{document}